\documentclass[letterpaper, 12pt]{article}
\usepackage[left=2cm,top=3cm,right=2cm]{geometry}
\usepackage{amsmath,amsfonts,amsthm}
\usepackage{amssymb}
\usepackage[all]{xy}
\usepackage{fancyhdr}
\usepackage{enumerate}
\usepackage{color}
\usepackage{caption}
\usepackage{graphicx, subfig}
\theoremstyle{definition}
\usepackage{mathtools}
\usepackage{amsfonts}
\usepackage{xypic}
\usepackage{xcolor}
\usepackage{accents}
\usepackage{float} 
\usepackage{hyperref}
\hypersetup{
    colorlinks=true,
    linkcolor=magenta,
    filecolor=red,      
    urlcolor=cyan,
    citecolor=green,
}
\numberwithin{equation}{section}
\newtheorem{theorem}{Theorem}[section]

\newtheorem{corollary}[theorem]{Corollary}
\newtheorem{proposition}[theorem]{Proposition}
\newtheorem{definition}[theorem]{Definition}

\newtheorem{conjecture}[theorem]{Conjecture}

\newcommand{\measurerestr}{ \,\raisebox{-.127ex}{\reflectbox{\rotatebox[origin=br]{-90}{$\lnot$}}}\, }

\title{\textbf{Bounded Diameter Under Mean Curvature Flow}}
\author{Wenkui Du}
\date{}

\begin{document}

\maketitle

\begin{abstract}
We prove that for the mean curvature flow of closed embedded hypersurfaces, the intrinsic diameter stays uniformly bounded as the flow   approaches the first singular time, provided all singularities are of neck or conical type. In particular, assuming Ilmanen’s multiplicity one conjecture and no cylinder conjecture, we conclude that in the two-dimensional case the diameter always stays bounded. We also obtain sharp $L^{n-1}$ bound for the curvature. The key ingredients for our proof are the Lojasiewicz inequalities by Colding-Minicozzi and Chodosh-Schulze, and the solution of the mean-convex neighborhood conjecture by Choi, Haslhofer, Hershkovits and White. Our results improve the prior results by Gianniotis-Haslhofer, where diameter and curvature control has been obtained under the more restrictive assumption that the flow is globally two-convex.
\end{abstract}

\tableofcontents

\section{Introduction}
For a family of closed embedded hypersurfaces $\mathcal{M}=\{M_{t}\}_{t\in [0,T)}$ in $\mathbb{R}^{n+1}$ evolving under mean curvature flow,  the first singular time $T<\infty$ is characterized by 
\begin{equation}
    \lim_{t\nearrow T}\max_{M_{t}}|A|=\infty,
\end{equation}
where $|A|$ denotes the norm of second fundamental form.\\\\
The central topic for mean curvature flow is to capture the geometric information of singularities.  One naturally wonders if one can control the geometry of surfaces evolving under mean curvature flow near the singular time. In particular, we have the following well-known conjecture:
\begin{conjecture}[bounded diameter conjecture]\label{conj}
For the mean curvature flow of closed embedded hypersurfaces $\mathcal{M}=\{M_{t}\}_{t\in [0,T)}$, the intrinsic diameter
stays uniformly bounded as the flow approaches to the first singular time $T$, i.e.
$$
\sup_{t\in [0,T)} \text{diam} (M_{t}, d_{t})<+\infty,
$$
where $d_{t}$ is the geodesic distance on the hypersurface $M_{t}$.
\end{conjecture}
\noindent We note that while the extrinsic diameter obviously stays bounded, controlling the intrinsic diameter is much more delicate. For example, one has to exclude the existence of fractal-like necks.  \\\\
\noindent The bounded diameter conjecture is also related to the question of establishing sharp integral curvature bounds. In fact, Topping \cite{Top08} proved that 
\begin{equation}
    \text{diam}(M_{t}, d_{t})\leq C_{n}\int_{M_{t}}H^{n-1}d\mu.
\end{equation}
Recently, Gianniotis-Haslhofer \cite{GH17} proved the bounded diameter conjecture under the assumption that the flow is globally two-convex, i.e. when the sum of any two principal curvatures of each hypersurface is positive.
\begin{theorem}[{Gianniotis-Haslhofer, \cite[Thm 1.1, Thm 1.2]{GH17}}]
If $\{M_{t}\subset \mathbb{R}^{n+1}\}_{t\in[0, T)}$ is a mean curvature flow of two-convex closed embedded hypersurfaces, then 
\begin{equation}
    \text{diam}(M_{t}, d_{t})\leq C,
\end{equation}
and moreover, 
\begin{equation}
    \int_{M_{t}}|A|^{n-1}d\mu < C,
\end{equation}
where $C$ only depends on certain geometric parameters of the initial hypersurface $M_{0}$.
\end{theorem}
\noindent Their proof uses the Lojasiewicz inequality from Colding-Minicozzi \cite{CM15} and the canonical neighborhood theorem from  Haslhofer-Kleiner \cite{HK17}. In particular, their results improved the earlier $L^{n-1-\varepsilon}$- curvature bound from Head \cite{Hea13} and Cheeger-Haslhofer-Naber \cite{CHN13}. In fact, the diameter bound and the $L^{n-1}$ curvature bound without $\varepsilon$ depend on the fine structure of singularities. Roughly speaking, the idea is that by the canonical neighborhood theorem, the high curvature regions look like necks or caps, and the Lojasiewicz inequality controls the tilting of the necks. The $L^{n-1}$ curvature  bound can fail if one removes the two-convexity assumption. For instance, it fails if $M_{0}=S^{n-2}_{r}\times S^{2}_{R}$ is a thin rotationally symmetric torus.  
\subsection{Main results}
In this paper, we generalize the diameter and curvature bounds from Gianniotis-Haslhofer \cite{GH17}. Instead of the global two-convexity assumption, we only impose an infinitesimal assumption on the structure of singularities. To describe this, let us first recall the notion of a tangent flow. For $X_{0}=(x_{0}, t_{0})$ in $\mathcal{M}$ and $\lambda>0$, we denote by $D_{\lambda}(\mathcal{M}-X_{0})$ the flow which is obtained by shifting $X_{0}$ to space-time origin and parabolically dilating by $\lambda>0$. Now, given any sequence $\lambda_{i}\rightarrow +\infty$, by Brakke's compactness theorem  \cite[Thm 7.1]{Ilm94}, one can always pass to a subsequential limit $\mathcal{M}^{\infty}$ of $D_{\lambda_{i}}(\mathcal{M}-X_{0})$. Any such limit $\mathcal{M}^{\infty}$ is called a tangent flow at $X_{0}$.  From Huisken's monontonicity formula  \cite[Thm 3.1]{Hui90}, all tangent flows are self-similarly shrinking ancient flows. \\\\
If the tangent flow is compact, one can verify the above Conjecture \ref{conj} easily from the uniqueness result of compact tangent flow of Schulze \cite{Sch14}. In the non-compact tangent flow case, one typically sees  cylindrical singularities or conical singularities. Because the $L^{n-1}$ curvature bound can fail for $S^{n-k}\times \mathbb{R}^{k}$ singularities, where $k\geq 2$, as we discussed in the end of previous section, we assume all cylindrical singularities are of neck type, i.e. $k=1$. The precise definition of neck and conical singularities is as follows:
\begin{definition}[singularities of neck or conical type]\label{neck conical} We say that a singularity of mean curvature flow at a space-time point $X_{0}=(x_{0}, t_{0})$ is of 
\begin{itemize}
    \item \emph{neck type}, if some tangent flow at $X_{0}$ is a one $\mathbb{R}$-factor cylindrical flow $\left\{S^{n-1}_{\sqrt{-2(n-1)t}}\times \mathbb{R}\right\}_{t<0}$ (up to a rotation) with multiplicity one.
    \item \emph{conical type}, if some tangent flow at $X_{0}$  is $\left\{\sqrt{-t}\Sigma\right\}_{t<0}$ with multiplicity one, where $\Sigma$ is an asymptotically conical shrinker.
\end{itemize}
\end{definition}
\noindent We recall that an asymptotically conical shrinker is a smooth hypersurface $\Sigma$ that satisfies
$$
H_{\Sigma}=\frac{x\cdot\nu_{\Sigma}}{2},
$$
and 
$$\lim_{t \nearrow 0} \sqrt{-t} \Sigma=\mathcal{C}.
$$
Here the convergence is in $C^{\infty}_{loc}(\mathbb{R}^{n+1}-\{0\})$  with multiplicity one, and $\mathcal{C}$ is a cone over a closed hypersurface $\Gamma^{n-1}\subset S^{n}\subset \mathbb{R}^{n+1}$. 
\noindent Now, we state our main result as follows:
\begin{theorem}[diameter bound]\label{main}
Let $\mathcal{M}=\{M_{t}\}_{t\in [0,T)}$  be a family of closed embedded hypersurfaces in $\mathbb{R}^{n+1}$ evolving under mean curvature flow with first singular time $T$.  If  every singularity at time $T$ is of neck type or conical type (see Definition \ref{neck conical}), then
$$
\sup_{t\in [0,T)} \text{diam} (M_{t}, d_{t})<+\infty.
$$
\end{theorem}
\noindent We also have the following sharp curvature bound.
\begin{theorem}[curvature bound]\label{cb}
Under the same assumptions as in Theorem \ref{main}, we have
$$
\sup_{t\in [0,T)} \int_{M_{t}}|A|^{n-1}d\mu < \infty.
$$
\end{theorem}
\noindent Theorem \ref{main} and Theorem \ref{cb} improve the main results of Gianniotis-Haslhofer \cite[Thm 1.1, Thm 1.2]{GH17} by removing the global two-convexity assumption in a favor of a much milder infinitesimal assumption on the structure of singularities.\\\\
Our conclusion is most striking for $n=2$, i.e. in the classical case of $2$-dimension surfaces in $\mathbb{R}^{3}$.
In fact, we obtain the bounded diameter conjecture in full generality assuming Ilmanen’s multiplicity one conjecture and no cylinder conjecture.  Let us recall these conjectures:
\begin{conjecture}[{Ilmanen, \cite[2. multiplicity one conjecture.]{Ilm03}}]
Let $M_{0}$ be a smooth embedded compact initial surface in $\mathbb{R}^{3}$ with  mean curvature flow $M_{t}$. Then a higher multiplicity plane cannot occur as a blowup limit of $M_{t}$ at the first singular time.
\end{conjecture}
\begin{conjecture}[{Ilmanen, \cite[12. No cylinder conjecture.]{Ilm03}}] Let $\Sigma$ be an embedded shrinking soliton in $\mathbb{R}^{3}$, and suppose that $\Sigma$ is not a round cylinder. Then, $\Sigma$ cannot have an end asymptotic to a cylinder. 
\end{conjecture} 
\noindent This reduction of assumptions is based on the following facts in the case where $n=2$. First, Ilmanen \cite{Ilm95} proved that for embedded mean curvature flow, every tangent flow is smoothly embedded. Then, Wang \cite{Wan14} proved that every non-compact tangent flow must have cylindrical and conical ends. Also, a more than  one $\mathbb{R}$-factor cylindrical flow is obviously excluded, so the above two conjectures imply that for $n=2$, we can only see singularities of neck type or conical type, or singularities whose blowing-up is compact tangent flow, Hence, we obtain the general conclusion for $n=2$:
\begin{corollary}[diameter and curvature bounds for $n=2$.]\label{main2}
Let $\mathcal{M}=\{M_{t}\}_{t\in [0,T)}$  be a family of closed embedded surfaces in $\mathbb{R}^{3}$ evolving under mean curvature flow with first singular time $T$, if we assume Ilmanen’s multiplicity one and no cylinder conjecture, then we have
$$
\sup_{t\in [0,T)} \text{diam} (M_{t}, d_{t})<+\infty,
$$
and
$$
\sup_{t\in [0,T)} \int_{M_{t}}|A|d\mu < \infty.
$$
\end{corollary}
\subsection{Ideas and key ingredients}
The main difficulty is to control the geodesic length in regions which contain singularities and have high curvature because diameter can be uniformly controlled  in the regions which have bounded geometry. To overcome the difficulty, we need to use the infinitesimal assumption on the structure of singularities to  decompose the flow  $\mathcal{M}$ into a low curvature part and a part that can be well approximated  in terms of certain standard geometric models. Since we only see singularities of cylindrical type and conical type, we can actually decompose the flow into the following more precise three parts (see Theorem \ref{decomposition}): 
\begin{itemize} 
\item  {low curvature part,}
\item  {mean convex part,}
\item {conical part.}
\end{itemize}
For obtaining this decomposition, the following two ingredients are important: the existence of mean convex canonical neighborhoods and a precise description of neighborhoods of conical singularities. \\\\ 
 First, based on the solutions to mean convex canonical neighborhood conjecture  proved by Choi, Haslhofer, Hershkovits  and White (see \cite{CHH18} and \cite{CHHW19}),  we obtain the existence of  mean convex canonical neighborhoods (see Theorem \ref{cylindrical}). For instance, when  $n=2$, we prove this by  contradiction using the classification of ancient low entropy flows (see \cite[Thm 1.2]{CHH18}).\\\\ 
 Then, by adapting the methods from Chodosh and Schulze's work of proving uniqueness of conical tangent flow in \cite{CS19},  we obtain  a precise description of neighborhoods of conical singularities (see Proposition \ref{cor2}). More precisely, we use their uniqueness result and obtain the estimates in a certain region of space-time with conical singularity as tip. Then, we need to extend the estimates into some parabolic ball with conical singularity as center. The strategies are applying  pseudolocality and Ecker-Huisken's curvature estimate to the renormalized flow and then rescale the estimates back.   \\\\
 Next, using the previous decomposition, we reduce to estimating diameter in tubes, i.e. regions entirely covered by necks (see Proposition \ref{reduce to neck}). This is related to  \cite[Prop 3.2]{GH17} but one difference here is that we need to use  Proposition \ref{cor2} to control the length of geodesics near conical singularities.\\\\ 
 For dealing with the estimation in  tubes, we need the following two ingredients: backwards stability (Proposition \ref{t}) and  small axis tilt (Proposition \ref{st}).  Since  bowl soliton and cylindrical flow are all the possible models in mean convex canonical neighborhoods (Theorem \ref{cylindrical}), we only need to prove backwards stability for these two models. On the other hand, motivated by the Lojasiewicz inequality and uniqueness argument of Colding and Minicozzi (see \cite{CM15}), we obtain the small axis tilt.\\\\
Then, using backwards stability and  small axis tilt, we show that in the neck region, the flow is close to tube flow with small tilt during a uniform period of time $\tau$ before the singular time $T$. Thus, we obtain the uniform control of the diameters in tubes, which by the above reductions yields a global diameter bound. This completes the sketch of the proof of Theorem \ref{main}. \\\\
Finally, for proving Theorem \ref{cb},  we estimate the curvature integral in all the three parts obtained by Theorem \ref{decomposition}. We show that mean convex part is the union of controlled number of tubes and caps, and Theorem \ref{main} implies the curvature bound in the tubes. The curvature bound in low curvature part, cap part, and conical part can be easily obtained. This completes the sketch of the proof of Theorem \ref{cb}. 
\subsection{Organization of the paper}
We organized the paper as follows: 
\begin{itemize}
\item In Section \ref{pre}, we introduce some general definitions and summarize necessary preliminaries. 
\item In Section \ref{mcx}, we discuss the canonical neighborhoods of neck singularities (Theorem \ref{cylindrical}).
\item In Section \ref{ncs}, we prove Proposition \ref{cor2}, which gives precise description of  the neighborhoods of conical singularities. 
\item In Section \ref{secd}, we  give Theorem \ref{decomposition}, the decomposition of the flow into low curvature part,  mean convex part and conical part. 
\item In Section \ref{rnc}, we reduce to estimating diameter in neck region (see Proposition \ref{reduce to neck}).
\item In Section \ref{bssat}, we discuss backwards stability (see Theorem \ref{t}) and small axis tilt (see Theorem \ref{st}). 
\item In Section \ref{completion}, we apply the above ingredients and conclude our proof of Theorem \ref{main} and Theorem \ref{cb}. 
\end{itemize}
\subsection*{Acknowledgements}
The author acknowledges his supervisor Robert Haslhofer for his patient guidance and  invaluable support in bringing this paper into fruition over the last year.
\section{Preliminaries}\label{pre}
In this section, we collect some general definitions and facts that will be used throughout the paper.
\begin{definition}[Brakke flow, {\cite[Def 6.2]{Ilm94}}]\label{brakke}
An $n$-dimensional integral Brakke flow $\mathcal{M}=\{\mu_{t}\}_{t\geq 0}$ is  a one parameter family of Radon measures  on $\mathbb{R}^{n+1}$ such that
\begin{enumerate}[(i)]
\item For almost every $t\geq 0$, $\mu_{t}=\mu_{V(t)}$, where $V(t)$ is an integral dimensional varifold  in $\mathbb{R}^{n+1}$, such that for every vector field $X$,
$$
\delta V_{t}(X)=-\int H \cdot Xd\mu_{t}
$$
holds for some vector valued function $H\in L^{2}(\mu_{t})$.
\item For every nonnegative function $f\in C_{c}^{1}(\mathbb{R}^{n+1}\times[a,b])$, where $0\leq a<b<\infty$, we have
$$\int f(\cdot,b)d\mu_{b}-\int f(\cdot,a)d\mu_{a}\leq \int_{a}^{b}\int -|H|^2f+H\cdot\nabla f+\frac{\partial f}{\partial t}d\mu_{t}dt. 
$$
\item For almost every $t\geq 0$, there are a $\mathbb{Z}$-valued function $\theta$ and a $n$-dimensional rectifiable set $M_{t}$, such that
$$
\mu_{t}=\theta\mathcal{H}^{n}\measurerestr M_{t}.
$$
\end{enumerate}
\end{definition}

\begin{definition}[convergence of Brakke flow, {\cite{Ilm94}}]
A sequence of Brakke flows $\mathcal{M}^{i}$ converges to $\mathcal{M}^{\infty}$, which is denoted by $\mathcal{M}^{i}\rightarrow \mathcal{M}^{\infty}$, if the following two properties hold.
\begin{enumerate}[(i)]
\item For all $t$ and $f\in C^{1}_{c}(\mathbb{R}^{n+1})$, we have\\
$$
\mu^{i}_{t}(f)\rightarrow \mu^{\infty}_{t}(f).
$$
\item For almost every $t$, there is a subsequence depending on $t$, such that
$$
V_{\mu^{i(t)}_{t}}\rightarrow V_{\mu^{\infty}_{t}}
$$
as varifolds.
\end{enumerate}
\end{definition}
\begin{definition}[entropy, \cite{CM12}]
  Let $\mu$ be $n$ dimensional integral rectifiable Radon measure.  The entropy of $\mu$ is given by 
  $$
  \mathrm{Ent}(\mu)=\sup_{x_{0}\in \mathbb{R}^{n+1},t_{0}>0}\int\frac{1}{({4\pi t_{0})}^{\frac{n}{2}}}\exp\left(-\frac{|x-x_{0}|^2}{4t_{0}}\right)d\mu.
  $$
  The entropy of Brakke flow $\mathcal{M}=\{\mu_{s}\}_{s\in[0,T)}$ is given by
  $$
  \mathrm{Ent}(\mathcal{M})=\sup_{s\in [0,T)}\mathrm{Ent}({\mu}_{s}).
  $$
\end{definition}
\noindent For mean curvature flow, we have following Huisken's monontonicity formula: 
\begin{theorem}[Huisken's monontonicity formula, {\cite[Thm 3.1]{Hui90}}]
Let $\mathcal{M}=\{M_{t}\}_{t\in [0,T)}$ in $\mathbb{R}^{n+1}$ be a family of closed smoothly embedded  hypersurfaces  evolving under mean curvature flow, $X_{0}=(x_{0},t_{0})$ be a space-time point, and
$$\Phi_{X_{0}}(x,t)=\frac{1}{(4\pi(t_{0}-t))^{\frac{n}{2}}}\exp\left({-\frac{|x-x_{0}|^2}{4(t_{0}-t)}}\right).$$
Then,
\begin{equation}\label{mon}
\frac{d}{dt} \int\Phi_{X_{0}}(x,t)d\mu_{t}=-\int \Phi_{X_{0}}(x,t)\left|H+\frac{(x-x_{0})\cdot \vec{n}}{2(t_{0}-t)}\right|^2 d\mu_{t}.
\end{equation}
If one replaces the equality in \eqref{mon} by an inequality, then it holds for Brakke flow.
\end{theorem}
\noindent As a consequence of Huisken's monotonicity formula, the Gaussian density can be defined as follows:
\begin{definition}[Gaussian density]
For Brakke flow $\mathcal{M}=\{\mu_{t}\}_{t\geq 0}$ and a space-time point $X_{0}=(x_{0},t_{0})$, the Gaussian density $\Theta(\mathcal{M},X_{0})$ of $\mathcal{M}$ at $X_{0}$ is given by
$$\Theta(\mathcal{M},X_{0})=\lim_{t\nearrow t_{0}}\int \Phi_{X_{0}}(x,t) d\mu_{t}.$$
\end{definition}
\noindent Now, we state Brakke's compactness theorem (formulated in terms of entropy):
\begin{theorem}[Brakke's compactness theorem, {\cite[Thm 7.1]{Ilm94}}]\label{compactness}
For a family of integral Brakke flows $\mathcal{M}^{i}$, if the entropy of $\mathcal{M}^{i}$ is uniformly bounded, then $\mathcal{M}^{i}\rightarrow\mathcal{M}^{\infty}$ (in Brakke flow's convergence) and  $\mathcal{M}^{\infty}$ is still an integral Brakke flow. 
\end{theorem}
\noindent The next two definitions are useful in our proof of diameter bound of mean curvature flow.
\begin{definition}[$\varepsilon$-close]\label{epsilon}
  Let $\mathcal{M}$ be a smooth mean curvature flow and  $X_{0}$ be a space-time point. $X\in P(X_{0},r)$ is $\varepsilon$-close to some mean curvature flow $\{\Sigma_{t}\}$ at $X_{0}$, if
  $D_{|H(X)|}(\mathcal{M}-X_{0})(t)$ is the $C^{[\varepsilon^{-1}]}$ graph of function $u(p,t)$ over $\Sigma_{t}\cap B(0,\varepsilon^{-1})$, and $$\|u(p,t)\|_{C^{[\varepsilon^{-1}]}}\leq\varepsilon$$
  holds for all $p\in \Sigma_{t}\cap B(0,\varepsilon^{-1})$ and for all $t\in (-\varepsilon^{-2},0)$. 
\end{definition}
\begin{definition}[regularity scale]\label{mr}
We define the regularity scale $R(X)$ of mean curvature flow $\mathcal{M}$ at $X=(x,t)$ as the supremum of $0\leq r\leq 1$, such that $M_{t'}\cap B(x,r)$ is a smooth graph for all $t'\in (t-r^{2},t+r^{2})$ and such that for all $Y=(y,s)$, where  $y\in B(x,r)$ and $s\in (t-r^{2}, t+r^2)$, we have
\begin{equation}
    |A(Y)|\leq\frac{1}{r}.
\end{equation}
\end{definition}
\noindent The important result related to regularity scale is White's local regularity theorem:
\begin{theorem} [{local regularity theorem}, \cite{white05}]\label{regular}
There are universal constants $\varepsilon(n)>0$ and $C(n)<\infty$ with the following significance: Let $\mathcal{M}$ be a smooth proper mean curvature flow in $U\times (t_{1},t_{0})$, $x_{0}\in U$. If the Gaussian density $\Theta(\mathcal{M},X_{0})<1+\varepsilon$, then there is a $\rho>0$, such that
$$
|A(x,t)|^2\leq \frac{C}{\rho^2}
$$
holds for every $x\in M_{t}\cap B_{\rho}(x_{0})$ and every $t\in (t_{0}-\rho^2, t_{0})$.
\end{theorem}
\noindent As a consequence of local regularity theorem, if a sequence of Brakke flows converges in the sense of Brakke flows and the limit is smooth, then the convergence is smooth.

\section{Mean convex canonical neighborhoods}\label{mcx}
In this section, we discuss the structure of the flow near neck singularities. By the solution of mean convex neighborhood conjecture in \cite[Thm 1.6]{CHH18} and \cite[Thm1.15]{CHHW19}, every neck singularity has a space-time neighborhood where the flow moves in one direction. More precisely, we have the following result.
\begin{theorem}[{mean convex canonical neighborhoods, c.f.  \cite{CHH18,CHHW19}}]\label{cylindrical}
  Let  $ \mathcal{M}=\{M_{t}\}_{t\in [0,T)}$ be a smooth mean curvature flow with first singular time $T$. Assume some tangent flow at $X_{0}=(x_{0},T)$ is a multiplicity one cylindrical flow with one $\mathbb{R}$-factor. Then, for every $\varepsilon>0$, there is a $\delta=\delta(X_{0},\varepsilon)>0$, such that any $X'\in B(x_{0},\delta)\times(T-\delta^{2},T)$ is $\varepsilon$-close (see Definition \ref{epsilon}) to either a one $\mathbb{R}$-factor cylinder shrinker or a bowl soliton. In particular, $\mathcal{M}$ is mean convex in the above neighborhood.
\end{theorem}
\noindent We recall that the bowl soliton is the unique (up to rigid motion and scaling) translating solution of the mean curvature flow that is rotationally symmetric and strictly convex, see \cite{AW94, CSS07, Has15}.\\\\ For $n\geq 3$, Theorem \ref{cylindrical} has been obtained in \cite[Cor 1.18]{CHHW19} as a consequence of classification of ancient asymptotically cylindrical flows. In a similar vein, we shall see that the statement for $n=2$ ultimately follows from the classification of ancient low entropy flows in \cite[Thm 1.2]{CHH18}. Since neither the statement about canonical neighborhoods nor its proof appeared in \cite{CHH18}, let us give a detailed proof here. To this end, we first recall  the  classification of ancient  low entropy  flows.
\begin{definition}[{ancient low entropy flow, \cite[Def 1.1]{CHH18}}] 
 An ancient low entropy  flow is an ancient,  unit regular, cyclic integral Brakke flow $\mathcal{M}$ in $\mathbb{R}^3$  with $\mathrm{Ent}(\mathcal{M})\leq \mathrm{Ent}(S^{1}\times \mathbb{R})$.
\end{definition}
\noindent Here, a Brakke flow is called ancient if it is defined on some interval starting from $-\infty$, and unit regular \cite{white05} if every space-time point with density $1$ is a regular point.
Being cyclic means, loosely speaking, that the flow has an inside and an outside, see \cite{white09} for the precise definition.
\begin{theorem}[{classification of ancient low entropy flow, \cite[Thm 1.2]{CHH18}}]\label{classification}
Every  ancient low entropy  flow $\mathcal{M}$ in $\mathbb{R}^3$ is one of the following: (i) static plane,  (ii) round shrinking sphere, (iii) ancient oval,   (iv) one $\mathbb{R}$-factor cylinder, (v) bowl soliton. 
\end{theorem}
\noindent Here, we recall that an ancient oval is ancient noncollapsed mean curvature flow of embedded 2-spheres that is not selfsimilar, see  \cite{white03,HH16} for existence and \cite{ADS18} for uniqueness.\\\\
By the discussion above, our task is to establish the existence of canonical neighborhoods in the case $n=2$. 
\begin{proof}[{Proof of Theorem \ref{cylindrical}}]
Suppose towards a contradiction, there exists $\varepsilon>0$, such that for all $\delta_{i}=\frac{1}{i}$,  we can find  $X_{i}=(x_{i},t_{i})$ in  $B(x_{0},\delta)\times(T-\delta^{2})$ that is neither $\varepsilon$-close to a one $\mathbb{R}$-factor cylinder nor to a bowl soliton.\\\\
Let $r_{i}=R(X_{i})$ be the regularity scale of $X_{i}$ (see Definition \ref{mr}). Because $X_{0}$ is singular, we have that $r_{i}$ converges to $0$. Let $\mathcal{M}^{i}=\mathcal{D}_{1/r_{i}}(\mathcal{M}-{X_{i}})$ be the flow which is obtained from $\mathcal{M}$ by shifting $X_{i}$ to the space-time origin and parabolically rescaling by $1/r_{i}$.  By the uniform boundedness of the entropy of $\mathcal{M}^{i}$ and  by Brakke's compactness theorem in \cite[Thm 7.1]{Ilm94}, we can pass to a subsequential limit $\mathcal{M}^{\infty}$, which is an integral Brakke flow.\\\\
\textbf{Claim.} $\mathcal{M}^{\infty}$ is an ancient low entropy flow.
\begin{proof}[{Proof of claim}]
By the above analysis and by \cite{white05} \cite[Thm 4.2]{white09}, we know that $\mathcal{M}^{\infty}$ is an integral, unit regular and cyclic Brakke flow. \\\\
By the definition of entropy of $\mathcal{M}^{\infty}$, and since the rescaled flows $\mathcal{M}^{i}$ converges to $\mathcal{M}^{\infty}$, for all $\eta>0$, we can find $(x_{i}', t_{i}')\rightarrow (x_{0},T)$ and $\rho_{i}\rightarrow 0$, such that
\begin{equation*}\label{eq1}
\mathrm{Ent}(\mathcal{M}^{\infty})-\eta<\int_{M_{t_{i}'-\rho^{2}_{i}}}\frac{1}{4\pi \rho^{2}_{i}}\mathrm{exp}\left(-\frac{|x-x_{i}'|^2}{4{\rho^{2}_{i}}}\right).
\end{equation*}
On the other hand, because $X_{0}$ is cylindrical singularity, 
using the upper-semicontinuity of Gaussian density and Huisken's monontoncity formula \cite[Thm 3.1]{Hui90}, we know there is some $\rho>0$, such that
\begin{equation*}\label{eq2}
\int_{M_{t-\rho^2}}\frac{1}{4\pi \rho^2}\mathrm{exp}\left(-\frac{|x-p|^2}{4r^2}\right)<\mathrm{Ent}(\mathbb{S}^{1}\times \mathbb{R})+\eta
\end{equation*}
holds for $|p-x_{0}|<\rho$ and $|t-T|<\rho^2$.\\\\
 By Huisken's monotonicity formula and the arbitrariness of $\eta>0$, we obtain
$$
 \mathrm{Ent}(\mathcal{M}^{\infty})\leq \mathrm{Ent}(\mathbb{S}^{1}\times \mathbb{R}).
$$
This implies that $\mathcal{M}^{\infty}$ is an ancient low entropy  flow, and thus proves the claim. 
\end{proof}
\noindent Hence,  $\mathcal{M}^{\infty}$ is from one of the five cases in  the classification of Theorem \ref{classification}. We will show that all cases yield a contradiction.\\\\
(i) If $\mathcal{M}^{\infty}$ is a static plane, by the local regularity theorem, we obtain that $r_{i}$ does not converge to $0$, which is a contradiction.\\\\
(ii) and (iii). If $\mathcal{M}^{\infty}$ is a round shrinking sphere or an ancient oval, by the local regularity theorem,  for $i$ large enough, $\mathcal{M}^{i}$ can be written as a graph of $C^{[1/\varepsilon]}$ function over $\mathcal{M}^{\infty}$ with $C^{[1/\varepsilon]}$ norm less than $\varepsilon$. This implies that for any given interval $[a, b]\subset (-\infty, 0)$, $\mathcal{M}^{i}$ is convex in $[a, b]$ for $i$ large enough if we choose $\varepsilon$ small.  Hence, ${M_{t}}$ becomes convex after some time $t$ close to $T$,  so Huisken's convergence theorem \cite{Hui84} implies that $\mathcal{M}$ becomes extinct as a round point. This  contradicts the assumption that $\mathcal{M}$ has cylindrical singularity at $X_{0}$.\\\\
(iv) and (v). Suppose $\mathcal{M}^{\infty}$ is a one $\mathbb{R}$-factor cylinder or a bowl soliton. First note that on the cylinder or bowl soliton, the regularity scale and the inverse mean curvature scale are comparable. Namely, there is a constant $C<+\infty$, such that,
\begin{equation}\label{RH}
C^{-1}H^{-1}(X)\leq R(X)\leq CH^{-1}(X)
\end{equation}
holds for all $X\in \mathcal{M}^{\infty}$. Recall that $\mathcal{M}^{i}=\mathcal{D}_{1/r_{i}}(\mathcal{M}-{X_{i}})$ was defined by rescaling by $1/r_{i}$, where $r_{i}=R(X_{i})$. Let  $\mathcal{\widetilde{M}}^{i}={D}_{H_{i}}(\mathcal{M}-X_{i})$, where we now rescale the flow by $H_{i}=H(X_{i})$. By the local regularity theorem, $\mathcal{M}^{i}$ converges to $\mathcal{M}^{\infty}$ smoothly. By this smooth convergence and the inequalities in \eqref{RH}, we get that
\begin{equation}\label{rhi}
    \frac{1}{2}C^{-1}H_{i}^{-1}\leq r_{i}\leq 2CH_{i}^{-1}
\end{equation}
holds for $i$ large enough. Therefore, we know that the two rescalings for $\mathcal{M}^{i}$ and $\mathcal{\widetilde{M}}^{i}$ only differ by a controlled factor. This implies that $\mathcal{\widetilde{M}}^{i}$ also converges to a cylinder or a bowl soliton smoothly. 
Correspondingly, we know that  $\mathcal{\widetilde{M}}^{i}$ is $\varepsilon$-close to a one $\mathbb{R}$- factor cylindrical flow or a bowl soliton for $i$ large enough. This contradicts our assumption that $X_{i}$ is neither $\varepsilon$-close to a one $\mathbb{R}$-factor cylinder nor to a bowl soliton. \\\\
By the above contradictions, we have completed the proof of the theorem.
\end{proof}
\section{Neighborhoods of conical singularities}\label{ncs}
In this section, we discuss the structure of the flow near conical singularities. In \cite{CS19}, Chodosh-Schulze proved that asymptotically conical tangent flows are unique and conical singularities are isolated.
\begin{theorem}[{uniqueness of asymptotically  conical tangent flows, \cite[Thm 1.1]{CS19}}]\label{cor1}
Let $\mathcal{M}$ be a mean curvature flow, and suppose that some tangent flow at $(x, T)$ is $\mathcal{M}_{\Sigma}=\{\sqrt{-t}\Sigma\}_{t<0}$ with multiplicity one, where $\Sigma$ is an asymptotically conical shrinker. Then, the tangent flow at $(x, T)$ is unique. Moreover, there exists  $\varepsilon>0$, such that
the flow is smooth in $B_{\varepsilon}(x)\times (T-\varepsilon^2, T]\setminus\{(x, T)\}$, and $M_{T}\cap B_{\varepsilon}(x)$ has a conical singularity at $x$ smoothly modeled on the asymptotic cone of $\Sigma$.
\end{theorem}
\noindent Here, we recall that an asymptotically conical shrinker $\Sigma$ is a smooth hypersurface which satisfies
$$
H_{\Sigma}=\frac{x\cdot\nu_{\Sigma}}{2},
$$
and 
$$\lim_{t \nearrow 0} \sqrt{-t} \Sigma=\mathcal{C}
$$
in $C^{\infty}_{loc}(\mathbb{R}^{n+1}-\{0\})$  with multiplicity one, where $\mathcal{C}$ is a cone over a closed hypersurface $\Gamma^{n-1}\subset S^{n}\subset \mathbb{R}^{n+1}$.\\\\
For our purpose, we need  the following more precise description of a neighborhood of a conical singularity:
\begin{proposition}\label{cor2}
Let $\mathcal{M}$ be a smooth mean curvature flow, and suppose that some tangent flow at $(x, T)$ is $\mathcal{M}_{\Sigma}=\{\sqrt{-t}\Sigma\}_{t<0}$ with multiplicity one, where $\Sigma$ is an  asymptotically conical shrinker. Then, for any $l\in \mathbb{N}_{+}$ and $b>0$,  we can find  $\varepsilon>0$, such that  for $t\in (T-\varepsilon^2,T)$, $M_{t}\cap B_{\varepsilon}(x)$ is a smooth graph of some function $u(t)$ over $\left(\sqrt{T-t}\Sigma+x\right)\cap B_{\varepsilon}(x)$ with
\begin{equation}\label{cs0}
    \|u(t)\|_{C^{l+1}}\leq b.
\end{equation}
\end{proposition}
\noindent For the proof, similarly as in Chodosh-Schulze \cite{CS19}, we will combine their uniqueness result (Theorem \ref{cor1})  with the pseudolocality theorem, which we now recall:
\begin{theorem}[{\cite[Theorem 1.5]{INS19}}, {\cite[Theorem 1.4]{CY07}}]\label{psl0}
Given $\delta>0$, there is $\gamma>0$ and $\rho<\infty$, such that if a mean curvature flow $\{M_{t}\}_{t\in [-1,0)}$ satisfies that $M_{-1}\cap B_{\rho}(0)$ is a Lipschitz graph over some region of the plane $P$ with Lipschitz constant smaller than $\gamma$ and $0\in M_{-1}$, then $M_{t}\cap B_{\rho}(0)$ intersects $B_{\delta}(0)$ and remains a $\delta$ Lipschitz graph within $B_{\delta}(0)$ over some region of the plane $P$ for all time $t\in [-1,0)$.   \end{theorem}
\noindent As a corollary, we have the following pseudolocality  for the renormalized flow $\widehat{\mathcal{M}}=\{\widehat{M}_{\tau}\}_{\tau\in [\tau', +\infty)}$, where $\widehat{M}_{\tau}=e^{\frac{\tau}{2}}M_{-e^{-\tau}}$ and $\tau=-\log(-t)$.
\begin{corollary}[pseudolocality for renormalized flow]\label{psl}
Given $\delta>0$, there is $\gamma>0$ and $\rho<\infty$, such that if the renormalized mean curvature flow $\{\widehat{M}_{\tau}\}_{\tau\in [\tau',+\infty)}$ satisfies that $\widehat{M}_{\tau'}\cap B_{e^{\frac{-\tau'}{2}}\rho}(0)$ is a Lipschitz graph over the plane $\{x_{n+1}=0\}$ with Lipschitz constant less than $\gamma$ and $0\in \widehat{M}_{\tau'}$, then $\widehat{M}_{\tau}\cap B_{e^{\frac{\tau-\tau'}{2}}\rho}(0)$ intersects $B_{e^{\frac{\tau-\tau'}{2}}\delta}(0)$ and remains a $\delta$ Lipschitz graph within $B_{e^{\frac{\tau-\tau'}{2}}\delta}(0)$ over the plane $\{x_{n+1}=0\}$ for all $\tau\in [\tau', +\infty)$.\\ \end{corollary}
\begin{proof}[{Proof of Proposition \ref{cor2}}]
Let $l\in \mathbb{N}_{+}$ and $b>0$ be fixed. For ease of notation, we suppose that $\mathcal{M}$ is defined on $[-1, 0)$ with the only conical singularity at $(0, 0)$.  Now, we consider the renormalized flow $\widehat{\mathcal{M}}=\{\widehat{M}_{\tau}\}_{\tau\in [\tau', +\infty)}$. By Theorem \ref{cor1} (uniqueness of conical tangent flow), $\widehat{M}_{\tau}$ converges to conical shrinker $\Sigma$ smoothly. Hence, for the given $l\in \mathbb{N}_{+}$, we can find monotone functions
$\rho(\tau)\rightarrow +\infty$, and $\sigma(\tau)\rightarrow 0$ as $\tau\rightarrow +\infty$, such that $\widehat{M}_{\tau}$ is a graph of some function $\hat{u}(\tau)$ over $\Sigma\cap B_{\rho(\tau)}$ with
\begin{equation}
    \|\hat{u}(\tau)\|_{C^{l+3}(B_{\rho(\tau)})}\leq \sigma(\tau),
\end{equation}
provided that $\tau$ is large enough.\\\\
Rescaling this back to the original flow, we see that 
$M_{t}\cap B_{\sqrt{|t|}\rho(\tau)}(0)$ is a graph of some function $u(t)$ over $\sqrt{-t}\Sigma$ for $|t|$ small enough. Moreover, for the given $b>0$ and $l\in \mathbb{N}_{+}$, we can find some $t_{1}=-e^{-\tau_{1}}$, where $\tau_{1}=\tau_{1}(b, l)$ is large enough, such that 
\begin{equation}\label{cs}
  \|u(t)\|_{C^{l+3}(B_{\sqrt{|t|}\rho(\tau_{1})})}
  \leq \|u(t)\|_{C^{l+3}(B_{\sqrt{|t|}\rho(\tau)})}\leq \sqrt{|t|}\sigma(\tau)\leq\sqrt{|t_{1}|}\sigma(\tau_{1})<b
\end{equation}
holds for all $t=-e^{-\tau}\in (t_{1}, 0)$.
Therefore, we have obtained the desired estimates \eqref{cs0} in the parabolic region $P=\{(x,t):|x|^2\leq \rho(\tau_{1})|t|, t\in [t_{1}, 0) \}$.\\\\ 
Next, we need to extend the estimates \eqref{cs0} from $P$ to some parabolic ball with center at $(0, 0)$. Fo the given $b$ and $t_{1}$ from above, let $\delta_{*}>0$ be a small constant to be fixed later. Then, $\gamma_{*}=\gamma_{*}(\frac{\delta_{*}}{2})<\frac{\delta_{*}}{2}$ and $\rho_{*}=\rho_{*}(\frac{\delta_{*}}{2})$ will be fixed according to Corollary \ref{psl} (pseudolality for renormalized flow). Because $\Sigma$ is an asymptotically conical shrinker, we can find $R_{1}=R_{1}(\Sigma, \gamma_{*}, \rho_{*})<\infty$ such that for $x\in \Sigma \cap B^{c}_{R_{1}}$, $\Sigma\cap B_{4\rho_{*}}(x)$ can be written as graph over $T_{x}\Sigma$ with $C^{1}$ norm less than $\frac{\gamma_{*}}{4}$.  Since $\widehat{M}_{\tau}$ converges to $\Sigma$ smoothly, for any $\tau_{2}=-\log(-t_{2})\geq \tau_{1}$, $\widehat{M}_{\tau_{2}}\cap B_{\rho(\tau_{2})}(0) \cap B^{c}_{R_{1}}(0)$ can be written as  graphs over $2\rho_{*}$-size balls on tangent planes of $\Sigma$ with $C^{1}$ norm less than $\frac{\gamma_{*}}{2}$,  provided that $\tau_{1}$ is large enough. \\\\
By Corollary \ref{psl} (pseudolality for renormalized flow), we see that for all $\tau\in [\tau_{2}, +\infty)$, $\widehat{M}_{\tau}\cap B_{\rho(\tau_{2})}(0)\cap  B^{c}_{R_{1}}(0)$ can be written as pieces of $\frac{\delta_{*}}{2}$ Lipschitz graphs over $e^{\frac{\tau-\tau_{2}}{2}}\frac{\delta_{*}}{2}$ size balls on tangent planes of $\Sigma$. 
By  Ecker-Huisken's curvature estimates for graphical flow in \cite[Thm 3.1, Thm 3.4]{EH91} (renormalized version), we see that $\widehat{M}_{\tau}\cap B_{\rho(\tau_{2})}(0)\cap  B^{c}_{R_{1}}(0)$ satisfies
\begin{equation}
    |\nabla^{l} A_{\widehat{M}_{\tau}}|\leq C\delta^{2}_{*} |\delta^2_{*}e^{\tau-\tau_{2}}|^{-\frac{l+1}{2}}\leq C\delta^{2}_{*},
\end{equation}
for all $\hat{x}\in \widehat{M}_{\tau}\cap B_{\rho(\tau_{2})}(0)\cap  B^{c}_{R_{1}}(0)$ and $\tau>\tau_{2}-\log(|t_{1}|-\omega)-2\log\delta_{*}$. Here, $0<\omega\ll |t_{1}|$ is a small constant, and $C<\infty$ is a constant depending on $\omega$ and the datum at time $\tau_{1}$. This implies that $\widehat{M}_{\tau}\cap B_{\rho(\tau_{2})}(0)\cap  B^{c}_{R_{1}}(0)$ is a graph of some function $\hat{u}(\tau)$ over $\Sigma$ with 
\begin{equation}
     \|\hat{u}(\tau)\|_{C^{l}(B_{\rho(\tau_{2})}\cap B^c_{R_{1}})}\leq C \delta^2_{*},
\end{equation}
for $\tau>\tau_{2}-\log(|t_{1}|-\omega)-2\log\delta_{*}$. \\\\
Now, we rescale this back to the original flow. Then, $M_t\cap B_{\sqrt{|t_{2}|}\rho(\tau_{2}))}\cap B^{c}_{\sqrt{|t_{2}|}R_{1}}$ can be written as graph of some function $u(t)$ over $\sqrt{-t}\Sigma$.
Note that $u(t)=\sqrt{|t|}\hat{u}(\tau)$  and $y=\sqrt{|t|}y'\in \sqrt{|t|}\Sigma$, we see that
\begin{equation}\label{deltaes}
     \|u(t)\|_{C^{l}(B_{\sqrt{|t_{2}|}\rho(\tau_{2}))}\cap B^{c}_{\sqrt{|t_{2}|}R_{1}})}\leq C \delta^2_{*}
\end{equation}
holds for $t\in (\delta^2_{*}(|t_{1}|-\omega)t_{2}, 0)$.\\\\
Now, we choose $\delta_{*}>0$ small enough, such that
\begin{equation}\label{deltab}
C\delta^2_{*}<b.
\end{equation}
Let $\varepsilon=\delta_{*}(|t_{1}|-\omega|)$. Noticing that $t_{2}$ is arbitrary and combing this with \eqref{deltaes}, \eqref{deltab} and \eqref{cs}, we obtain the estimation \eqref{cs0} in $B_{\varepsilon}(0)\times (-\varepsilon^2, 0)$.  This completes the proof of the theorem.\\\\
\end{proof}
\section{Decomposition of the flow}\label{secd}
The goal of this section is to decompose the flow into three parts: low curvature part, mean convex part and conical part. \\\\
Suppose that $\Omega$ is a compact domain in $\mathbb{R}^{n+1}$. Let $\mathcal{M}=\{M_{t}\}_{t\in [0,T)}\subset \Omega$  be a mean curvature flow of closed embedded hypersurfaces in $\mathbb{R}^{n+1}$ with first singular time $T$. Denote by $S_{T}(\mathcal{M})\subset \mathbb{R}^{n+1}$ the singular set at time $T$.   Assume that for each $x\in S_{T}(\mathcal{M})$,  some tangent flow  at  $(x,T)$ is a one $\mathbb{R}$-factor cylindrical or an  asymptotically conical shrinker with multiplicity one. 
\begin{theorem}[decomposition of flow]\label{decomposition} Under the above assumptions, for every $\varepsilon>0$,
there exist constants $\delta>\rho>0$, and a decomposition of the domain 
\begin{equation}\label{MCL}
    \Omega= \Omega_{M} \cup \Omega_{C} \cup \Omega_{L},
\end{equation}
such that the following statements hold:\\
\begin{enumerate}[(i)]
\item $\Omega_{M}$ is the union of finitely many balls,
\begin{equation}\label{N}
    \Omega_{M}=\mathop{\cup}\limits_{j=1}^{k} B_{\delta_{j}}(p_{j}),
\end{equation}
where  $\delta_{j}\geq \delta$ for $j=1,\dots, k$, such that
\begin{enumerate}[(a)]
\item For $j=1,\dots,k$, the flow $\{M_{t}\cap B_{\delta_{j}}(p_{j})\}_{t\in(T-\delta^2,T)}$ is mean convex.
\item Any $(p,t)\in \mathcal{M}\cap\left( \Omega_{M} \times(T-\delta^{2},T)\right)$ is $\varepsilon$-close (see Definition \ref{epsilon}) to either a round shrinking cylinder with one $\mathbb{R}$-factor or a translating bowl soliton. 
\end{enumerate}
\item $\Omega_{C}$ is the disjoint union of finitely many balls
\begin{equation}
    \Omega_{C}=\mathop{\cup}\limits_{i=1}^{m} B_{\delta}(x_{i}),
\end{equation}
such that for all $t\in(T-\delta^2, T)$ and $i\in\{1, \dots, m\}$, $ M_{t}\cap B_{\delta}(x_{i})$ can be written as $C^{l+1}$ graph over   $\left(\sqrt{T-t}\Sigma_{i}+x_{i}\right)\cap B_{\delta}(x_{i})$ with $C^{l+1}$ norm less than $\varepsilon$, where $\Sigma_{i}$ is an asymptotically conical shrinker.
\item For $t\in (T-\delta^2, T)$ and every $p\in M_{t}\cap \Omega_{L}$ or $p\in M_{t}\cap \Omega_{C}\setminus \mathop{\cup}\limits_{i=1}^{m}  B_{\frac{\delta}{2}}(x_{i})$, the regularity scale (see Definition \ref{mr})  at $(p,t)$  satisfies
\begin{equation}
    R(p,t)\geq \rho.
\end{equation}
\end{enumerate}
\end{theorem}
\begin{proof}
By the local regularity theorem (or Theorem \ref{regular}), we know that the set of all regular space-time points is open. This implies that the singular set $S_{T}(\mathcal{M})$ at first singular time $T$  is closed and bounded, hence it is compact.\\\\ 
By the isolated property of conical singularities from Theorem \ref{cor1}, we know there are only finitely many conical singularities $x_{1},\dots, x_{m}$ at time $T$. For each $i\in\{1, \dots, m\}$, we denote by $\{\sqrt{-t}\Sigma_{i}\}$ the tangent flow at the conical singularity $(x_{i},T)$. Now, according to Theorem \ref{cor1},  we can find $\delta'_{i}>0$, such that  $B_{\delta'_{i}}(x_{i})$ are disjoint, and for $t\in (T-\delta_{i}'^{2},T)$, $M_{t}\cap B_{\delta'_{i}}(x_{i})$ are smooth graphs over $\left(\sqrt{T-t}\Sigma_{i}+x_{i}\right)\cap B_{\delta'_{i}}(x_{i})$ with $C^{l+1}$ norm less than $\varepsilon$. \\\\
\noindent Hence, $S_{T}(\mathcal{M})\setminus \mathop{\cup}\limits_{i=1}^{m}\{x_{i}\}$ is still compact and all remaining singularities are one $\mathbb{R}$-factor cylindrical. Now, by Theorem \ref{cylindrical}, for each $p\in S_{T}(\mathcal{M})\setminus \mathop{\cup}\limits_{i=1}^{m}\{x_{i}\}$, we can find $\delta'>0$, such that the flow $\{M_{t}\cap B_{\delta'}(p)\}$ is mean convex and any $(p',t')\in \mathcal{M}\cap\left( B(p,\delta')\times(T-\delta'^{2},T)\right)$ is $\varepsilon$-close to either a cylindrical shrinker or a bowl soliton. By compactness of $S_{T}(\mathcal{M})\setminus \mathop{\cup}\limits_{i=1}^{m}\{x_{i}\}$, we can find finitely many cylindrical singularities $\{p_{1},\dots, p_{k}\}\subset S_{T}(\mathcal{M})\setminus \mathop{\cup}\limits_{i=1}^{m}\{x_{i}\}$ and $\delta_{j}$ corresponding to $p_{j}$ for $j\in\{1,\cdots, k\}$, 
such that the above properties in statement 2 hold and 
$S_{T}(\mathcal{M})\setminus \mathop{\cup}\limits_{i=1}^{m}\{x_{i}\}\subset \cup_{j=1}^{k} B_{\delta_{j}}(p_{j})$.\\\\
Now, we choose $\delta=\min\{\delta_{1},\cdots,\delta_{m},\delta'_{1},\cdots,\delta'_{k}\}$ and take 
$$\Omega_{M}=\mathop{\cup}\limits_{j=1}^{k} B_{\delta_{j}}(p_{j})$$
and 
$$ \Omega_{C}=\mathop{\cup}\limits_{i=1}^{m} B_{\delta}(x_{i}).$$
They satisfy the requirements in statement 1 and statement 2.\\\\
 Let $\Omega_{L}=\Omega\setminus(\Omega_{M}\cup \Omega_{C})$. Note that $\Omega_{L}$ and $\mathop{\cup}\limits_{i=1}^{m} \left( \overline{B_{\delta}(x_{i})}\setminus B_{\frac{\delta}{2}}(x_{i})\right)$ are compact sets. By the local regularity theorem, we know the regularity scale $R$ is  positive over their union. Combining this with the fact that regularity scale $R$ is $1$-Lipschitz, we can find some $\rho\in (0,\frac{\delta}{2})$, such that for 
$t\in (T-\delta^2, T)$ and every $p\in M_{t}\cap \Omega_{L}$ or $p\in M_{t}\cap \Omega_{C}\setminus \mathop{\cup}\limits_{i=1}^{m}  B_{\frac{\delta}{2}}(x_{i})$, we have
\begin{equation}
    R(p,t)\geq \rho.
\end{equation}
This completes the proof.
\end{proof}
\section{Reduction to the neck region}\label{rnc}
In this section, based on our decomposition from Section \ref{secd}, we will reduce to estimating the diameter in neck regions.\\\\
As in the Section \ref{secd}, let $\mathcal{M}=\{M_{t}\}_{t\in[0,T)}$ be a mean curvature flow of closed embedded hypersurfaces in $\mathbb{R}^{n+1}$ satisfying the assumptions of Theorem \ref{main}. Let $\Omega$ be a large ball that contains $M_{0}$. By Theorem \ref{decomposition}, for every $\varepsilon>0$, we can find constants $\delta>0, \rho>0$ and a decomposition
$$
    \Omega= \Omega_{M} \cup \Omega_{C} \cup \Omega_{L},
$$
into  a mean convex part $ \Omega_{M}$, a conical part $\Omega_{C}$, and a low curvature part $\Omega_{L}$.\\\\
First, we reduce to controlling the diameter in $\Omega_{M} \cup \Omega_{C} $. To this end, for any $\bar{t}\in (T-\delta^2,T)$, we consider
\begin{equation}
D(M_{\bar{t}}):=\sup\left\{l(\gamma):\gamma\, \text{is a minimizing geodesic in}\, (M_{\bar{t}}, d_{\bar{t}})\, \text{and}\, R<\frac{\rho}{2}\,  \text{along} \, \gamma \right\},
\end{equation}
where $R$ denotes the regularity scale (see Definition \ref{mr}).\\\\
 From Theorem \ref{decomposition}, we know the geodesic $\gamma$ in the above definition of $D(M_{\bar{t}})$ is contained in the region $\Omega_{M} \cup \Omega_{C}$. The next proposition reduces  to controlling the diameter in $\Omega_{M} \cup \Omega_{C}$.
\begin{proposition}
 There exists a constant $C<\infty$, such that
\begin{equation}
\text{diam}(M_{\bar{t}})\leq C(1+D(M_{\bar{t}})).
\end{equation}
\end{proposition}
\begin{proof}
Let $\gamma:[0,L] \rightarrow  (M_{\bar{t}}, d_{\bar{t}})$ be a minimizing geodesic parametrized by arclength. We choose a maximal collection $s_{1}, \dots, s_{N} \in [0, L]$, such that $R(\gamma(s_{i}),\bar{t})\geq \frac{\rho}{2}$ and $|s_{i}-s_{j}|\geq 1$ for $i\in\{1, \dots, N\}$.  Let $B_{i}$ to a geodesic ball with center $\gamma(s_{i})$ and radius $\frac{1}{2}$. These balls are disjoint. Because the regularity scale $R$ is bounded below by $\rho/2$ at $(s_{i}, \bar{t})$, we have a uniform bound on the second fundamental form within  balls of definite size with center $\gamma(s_{i})$.  By Gauss-Codazzi equation, the sectional curvatures are uniformly bounded within these balls of definite size. According to volume comparison, we get a uniform lower bound $c(\rho)$ for the volume of $B_{i}$, i.e
$$
\mathcal{H}^{n}(B_{i})\geq c(\rho).
$$
Denoting by $\mathcal{A}$  the area of initial surface of the flow, we have
$$
N\leq\frac{\mathcal{A}}{c(\rho)}=:C_{0}.
$$
Now, we estimate the length of $\gamma$ by adding the length of the $N$ pieces $\gamma([s_{i}-\frac{1}{2},s_{i}+\frac{1}{2}])$ for $i\in \{1,\dots, N\}$, and the other $N+1$  disjoint arcs. Since $N\leq C_{0}$, we conclude that 
$$
\text{diam}(M_{\bar{t}})\leq C_{0}+(C_{0}+1)D(M_{\bar{t}}).
$$
Setting  $C=C_{0}+1$, this proves the proposition.
\end{proof}
\noindent Next, we reduce to controlling the diameter in $\Omega_{M}$. We define
$$
D'(M_{\bar{t}}):=\sup\left\{l(\gamma):\gamma\, \text{is a minimizing geodesic in}\, (M_{\bar{t}}, d_{\bar{t}})\, \text{and}\, R<\frac{\rho}{2}\,  \text{along} \, \gamma\, \text{and}\, \gamma\subset \Omega_{M}\right\}
$$
Then, we have the following proposition.
\begin{proposition}\label{dconical}
There exists a constant $C<\infty$, such that
\begin{equation}
D(M_{\bar{t}})\leq C(D'(M_{\bar{t}})+1).   
\end{equation}
\end{proposition}
\begin{proof}
For ease of notation, we first analyze the case where the flow has only one conical singularity. After a translation in space-time, we can assume that this singularity is at $(0,0)$. Let $\Sigma$ be the time $-1$ slice of the tangent flow at this point.  By Theorem \ref{decomposition}, $M_{\bar{t}}\cap B_{\delta}(0)$ is a $C^{l+1}$ graph of function $u$ over $(-\bar{t})^{\frac{1}{2}}\Sigma$ with $C^{l+1}$ norm less than $\varepsilon$. We can pushforward the metric of  $(-\bar{t})^{\frac{1}{2}}\Sigma$ to $M_{\bar{t}}$  via $u$. Then, the  original metric and the pushforward metric of $M_{\bar{t}}\cap B_{\delta}(0)$ are uniformly equivalent for $\bar{t}\in (-\delta^2, 0)$.
Since $\Sigma$ is an asymptotically  conical shrinker, the family of metrics of $(-\bar{t})^{\frac{1}{2}}\Sigma\cap B_{\delta}(0)$ is uniformly bounded. Hence, the family of original metrics on $M_{\bar{t}}\cap B_{\delta}(0)$ is uniformly bounded. This  implies that $l(\gamma_{\bar{t}})$, the length of geodesics $\gamma_{\bar{t}}\subset M_{\bar{t}}\cap B_{\delta}(0)$, is uniformly bounded by some constant $C(\Sigma, \varepsilon)$ for $\bar{t}\in (-\delta^{2}, 0)$.\\\\
The argument for finitely many conical singularities case is similar. Hence, for $\bar{t}\in (-\delta^2,0)$, we obtain
$$
D(M_{\bar{t}}) \leq (m+1) D'(M_{\bar{t}})+ \sum_{i=1}^{m}C(\Sigma_{i},\varepsilon),
$$
where $C(\Sigma_{i},\varepsilon)$ is a constant depending on the asymptotically conical shrinker $\Sigma_{i}$ at $x_{i}$ as in Theorem \ref{decomposition}. Choosing $ C=m+1+\sum_{i=1}^{m}C(\Sigma_{i},\varepsilon)$,  this proves the proposition.
\end{proof}
\noindent The next step is to reduce to controlling the diameter in neck regions. We first recall the definition of strong $\varepsilon$-neck and very strong $\varepsilon$-neck.
\begin{definition}[strong $\varepsilon$-neck and very strong $\varepsilon$-neck]\label{str}
A mean curvature flow $\mathcal{M}$ is said to have a strong $\varepsilon$-neck with center $p$ and radius $r$ at time $t_{0}$ if the rescaled flow  $\{r^{-1}(M_{t_{0}+r^2t}-p)\}_{t\in(-1,0]}$ is $\varepsilon$-close in $C^{\lfloor\frac{1}{\varepsilon}\rfloor}$ sense in $B_{\varepsilon^{-1}}(0)\times (-1, 0]$  to
$\{O_{p}(S^{n-1}(\sqrt{1-2(n-1)t})\times \mathbb{R})\}_{t\in(-1,0]}$ for some  $O_{p}\in SO(n)$.  If  we can replace the interval $(-1,0]$ by $(-2\mathcal{T},0]$, where $\mathcal{T}$ is from Proposition \ref{t}, then we say that $\mathcal{M}$ has a very strong $\varepsilon$-neck.
\end{definition}
\begin{definition}[$\varepsilon$-tube]\label{tube}
$N\subset M_{\bar{t}}$ is called an $\varepsilon$-tube if $N$ is diffeomorphic to a cylinder, and each $x\in N$ lies on the central sphere of a very strong $\varepsilon$-neck (see Definition \ref{str}) of $\mathcal{M}$ with radius $(n-1)H^{-1}(x)$ at time $\bar{t}$.
\end{definition}
\noindent Now, we define
$$
L(M_{\bar{t}}):=\sup \left\{ \text{diam}(N, d_{\bar{t}}): N\subset M_{\bar{t}} \, \text{is an} \,\varepsilon\text{-}\text{tube}\, \text{with regularity scale}\, R<\frac{\rho}{2}\,\text{and}\,  N\subset \Omega_{M} \right\}.
$$
Then, we have the following proposition.
\begin{proposition}[reduction to neck region]\label{reduce to neck}
There exists a constant $C<\infty$, such that 
\begin{equation}
 D'(M_{\bar{t}})\leq  C(L(M_{\bar{t}})+1)  
\end{equation}
\end{proposition}
\begin{proof}
Let $\gamma$ be a minimizing geodesic in $M_{\bar{t}}$ with $R<\frac{\rho}{2}$ along $\gamma$, and $\gamma\subset\Omega_{M}$. According to Theorem \ref{decomposition}, we know that for $\varepsilon>0$ small enough, $\gamma$ is contained in an $\varepsilon$-tube possibly with caps as ends, or with its ends identified.  Notice that the mean curvatures of the points on the caps are bounded below by $C\rho^{-1}$. Hence, the caps have diameter bounded by $C(\varepsilon)\rho^{-1}$. If the ends of $\varepsilon$-tube are identified, we only need to remove small controlled pieces and reduce the argument to $\varepsilon$-tube case. This implies the assertion.
\end{proof}
\section{Backwards stability and small axis tilt}\label{bssat}
In this section, we prove Proposition \ref{t} (backwards stability) and Proposition \ref{st} (small axis tilt). 
\noindent Our backwards stability for necks on the bowl soliton is a special case of what has been observed in more general context for Ricci flow by Kleiner-Lott in \cite{KL17}. 
\begin{proposition}[backwards stability]\label{t}
For all $\delta_{0}>0$ and $\delta_{1}>0$ small enough, we can find  $\mathcal{T}=\mathcal{T}(\delta_{0}, \delta_{1})<\infty$ with the following property. Suppose $\mathcal{M}$ is a cylindrical flow 
or a translating bowl, and $\mathcal{M}$ has a strong $\delta_{0}$-neck (see Definition \ref{str}) with center $p$ and radius $\sqrt{2(n-1)}$ at time $-1$. Then, for all $t\in (-\infty, \mathcal{T}]$,  the flow $\mathcal{M}$ has a strong $\delta_{1}$-neck with center $p$ and radius $\sqrt{2(n-1)|t|}$ at time $t$.
\end{proposition}
\begin{proof}
If $\mathcal{M}$ is a cylindrical flow, after a rotation, we have
$$
M_{t}=\left(\sqrt{2(n-1)(t_{*}-t)}S^{n-1}+p_{*}\right)\times \mathbb{R},
$$
for some $t_{*}\in \mathbb{R}$ and $p_{*}\in  \mathbb{R}^{n+1} $.\\\\
Because  $\mathcal{M}$ has a strong $\delta_{0}$-neck with center $p$ and radius $\sqrt{2(n-1)}$ at time $-1$, possibly after shifting $p_{*}$ along the $x_{n+1}$-axis, we get that
\begin{equation}\label{t*}
 |t_{*}|=O(\delta_{0}) \quad |p_{*}-p|=O(\delta_{0}).
\end{equation}
Now, given any $t<0$, and $s\in (-1, 0]$, we compute 
$$
\frac{1}{\sqrt{2(n-1)|t|}}\left(M_{(t+{2(n-1)|t|s)}}-p\right)=\left(\sqrt{1-2(n-1)s+|t|^{-1}t_{*}}S^{n-1}+\frac{({p_{*}-p})}{\sqrt{2(n-1)|t|}}\right)\times \mathbb{R}.
$$
Thanks to \eqref{t*}, the terms $|t_{*}t^{-1}|$ and $|t|^{-\frac{1}{2}}|p-p_{*}|$ can be made arbitrarily small for $|t|$ large enough. This shows that $\mathcal{M}$ has a strong $\delta_{1}$-neck with radius $\sqrt{2(n-1)|t|}$ at time $t$, provided $t$ is negative enough (depending only on $\delta_{0}$, $\delta_{1}$).
\\\\
If $\mathcal{M}$ is  translating bowl, up to a rotation and translation, we can assume that $M_{t}=M_{0}+ct\vec{e}_{n+1}$, where $M_{0}$ is the graph of function $\psi(x)=\varphi(|x|)$, and $\varphi$ is strictly convex and attains its minimum at the origin and has the following asymptotic expansion as $r\rightarrow +\infty$ (see  \cite[Lem 2.2]{CSS07}):
\begin{equation}
\varphi(r)=\frac{cr^2}{2(n-1)}+O(\log \sqrt{c}r).    
\end{equation}
The function $\varphi(r)$ is strictly monotone on $[0,+\infty)$ and has an inverse $r(h)$, where
\begin{equation}\label{r(h)}
    r(h)=\sqrt{2(n-1)c^{-1}h}+o\left(\frac{1}{\sqrt{ch}}\right).
\end{equation}
Because $\mathcal{M}$ has a (strong) $\delta_{0}$-neck at time $-1$ with radius $\sqrt{2(n-1)}$ and center $p=(x_{p},z_{p})$, setting $h=c+z_{p}$, we obtain
\begin{equation}
    |x_{p}|=|r(h)-\sqrt{2(n-1)}|=O(\delta_{0}), \quad \quad
    |r(h\pm\delta^{-1}_{0})-\sqrt{2(n-1)}|=O(\delta_{0}).
\end{equation}
This implies
\begin{equation}\label{est}
    \left|\frac{z_{p}}{c}\right|=O(\delta_{0}),\quad \quad \frac{1}{c}=O(\delta^2_{0}).
\end{equation}
Now, for $|z|\leq \sqrt{2(n-1)|t|}\delta^{-1}_{1}$ and $s\in (-1, 0]$, using \eqref{r(h)} and \eqref{est}, we compute,
\begin{equation}
    \frac{r\left(z+z_{p}+c(t+2(n-1)|t|s)\right)}{\sqrt{2(n-1)|t|}}=\sqrt{-1+2(n-1)s+|t|^{-1/2}\delta_{1}^{-1}O(\delta^{-2}_{0})}+|t|^{-1/2}O(\delta^{-2}_{0}\delta^{-1/2}_{1}).
\end{equation}
Therefore, for any  $\delta_{0}$ and $\delta_{1}$ small enough, we can find $\mathcal{T}>0$ large enough, such that if $t\leq-\mathcal{T}$, then $\mathcal{M}$ has a strong $\delta_{1}$-neck with radius $\sqrt{2(n-1)|t|}$ at time $t$. This completes the proof.
\end{proof}
\noindent Next, we state the proposition about small axis tilt.
\begin{proposition}[[{small axis tilt, \cite[Prop 4.1]{GH17}}]\label{st}
Let $\mathcal{M}$ be a mean curvature flow with entropy bound $\Lambda$. For all $\varepsilon_{0}>0$, there exists $\varepsilon_{1}=\varepsilon_{1}(\varepsilon_{0}, \Lambda)>0$, such that if $\mathcal{M}$ has a strong $\varepsilon_{1}$-neck with center $p$ and radius $\sqrt{2(n-1)(t_{*}-t)}$ for all time $t\in [t_{0},t_{1}]$, then $\mathcal{M}$ has a strong $\varepsilon_{0}$-neck with center $p$, radius $\sqrt{2(n-1)(t_{*}-t)}$ and a fixed direction $v$ as axis for all time $t\in [t_{0},t_{1}]$.
\end{proposition}
\noindent This has been proved proof  in \cite[Prop 4.1]{GH17}, using the  Lojasiewicz inequality from Coding-Minicozzi \cite{CM15}. Since this is a crucial ingredient for establishing our diameter bound, we include the proof here as well.
\begin{proof}
Consider the renormalized mean curvature flow $\Sigma_{s}=\frac{1}{\sqrt{t_{*}-t}}M_{t}$, where  $ s=s(t)=-\log(t_{*}-t)$. If $\varepsilon_{1}>0$ is small enough, then we can apply  \cite[Thm 6.1]{CM15}, which gives
\begin{equation}
|F(\Sigma_{s})-F(Z)|^{1+\mu}\leq K(F(\Sigma_{s-1})-F(\Sigma_{s+1})),
\end{equation}
for $s\in [s_{0}, s_{1}]:=[s(t_{0}), s(t_{1}))]$ and some cylinder $Z=\sqrt{2(n-1)}S^{n-1}\times\mathbb{R}$. Here, $\mu>0, K<+\infty$ are constants, and $F$-functional is defined by
\begin{equation}
F(\Sigma)=\int_{\Sigma}\frac{1}{(4\pi)^{\frac{n}{2}} }e^{-\frac{|x|^2}{4}}.
\end{equation}
Applying the discrete Lojasiewicz lemma from \cite[Lem 6.9]{CM15} (see also  \cite[A.1]{GH17}), we infer that for every $\varepsilon>0$, as long as $\varepsilon_{1}>0$  is small enough, we have 
\begin{equation}
\sum_{j=s_{0}+1}^{s_{1}}\left(F(\Sigma_{j})-F(\Sigma_{j+1})\right)^{\frac{1}{2}}<\varepsilon.
\end{equation}
By the Cauchy-Schwarz inequality and Huisken's monontonicty formula, we have
\begin{equation}
\int_{s_{0}}^{s_{1}}\int_{\Sigma_{s}}\left|H+\frac{1}{2}x^{\perp}\right|\frac{e^{-\frac{|x|^2}{4}}}{(4\pi)^{\frac{n}{2}}}\leq \Lambda^{\frac{1}{2}}\sum_{j=s_{0}+1}^{s_{1}}(F(\Sigma_{j})-F(\Sigma_{j+1}))^{\frac{1}{2}}<\Lambda^{\frac{1}{2}}\varepsilon.
\end{equation}
Using also \cite[Lemma A.48]{CM15} and interpolation, this yields the assertion of Proposition \ref{st}. 
\end{proof}

\section{Completion of the proof}\label{completion}
In this final section, we complete the proof of our  Theorem \ref{main} and  Theorem \ref{cb}.
\begin{proof}[Proof of Theorem \ref{main}]
Let $\mathcal{M}=\{M_{t}\}_{t\in [0,T)}$ be a mean curvature flow of closed embedded hypersurfaces in $\mathbb{R}^{n+1}$ with first singular time $T$, satisfying the assumptions of the theorem. Namely, for every singular point at time $T$, there is some tangent flow which is either a one $\mathbb{R}$-factor cylindrical flow or an asymptotically conical flow with multiplicity one.\\\\
The following argument depends on various neck-quality parameters $0<\varepsilon \ll \varepsilon_{1} \ll \varepsilon_{0}$. The logical order for choosing these constants is that  one first fixes $\varepsilon_{0}>0$ small enough depending only on the dimension,  then lets $\varepsilon_{1}=\varepsilon_{1}(\Lambda,\varepsilon_{0})>0$ be the constant from Proposition \ref{st} and finally chooses   $0<\varepsilon \ll \varepsilon_{1}$ given by the claim below. Given any $\bar{t}\in (T-\delta^2, T)$, where $\delta=\delta(\varepsilon)>0$ is from Theorem \ref{decomposition} (decomposition of flow), we want to show that there is a constant $C<\infty$ independent on $\bar{t}$, such that
$$
\text{diam}(M_{\bar{t}})\leq C.
$$
By the Theorem \ref{decomposition} (decomposition of flow) and  Proposition \ref{reduce to neck} (reduction to neck region), it is enough to estimate 
$$
L(M_{\bar{t}}):=\sup \left\{ \text{diam}(N, d_{\bar{t}}): N\subset M_{\bar{t}} \, \text{is an} \,\varepsilon\text{-}\text{tube}\, \text{with regularity scale}\, R<\frac{\rho}{2}\,\text{and}\,  N\subset \Omega_{M} \right\}.
$$
Here, $R$ is regularity scale (see Definition \ref{mr}),   $\rho=\rho(\varepsilon)\in (0, \frac{\delta}{2})$ is from  Theorem \ref{decomposition}, and $N\subset M_{\bar{t}}$ is an $\varepsilon$-tube (see Definition \ref{tube}).  
For each $\varepsilon$-tube $N$, we can find an $\varepsilon$-approximate central curve $\gamma$ parametrized by arclength, such that for each $p\in \gamma$, $\partial_{s}\gamma(p)$ determines the axis of the $\varepsilon$-neck centered at $p$ (see \cite{BHH17}). Then, we only need to estimate the length of $\gamma$.\\\\  Note that for any $p\in\gamma$ and $x\in N$ in the central sphere of the very strong $\varepsilon$-neck with center $p$ and radius $r_{p}$ at time $\bar{t}$, and by the definition of regularity scale in Definition \ref{mr}, for $\varepsilon>0$ small enough, we have that 
\begin{equation}\label{r_p}
    r_{p}\leq \frac{1}{2}\rho.
\end{equation} 
Let
\begin{equation}\label{tau'}
    \tau=\frac{1}{200(n-1)}\rho^{2}.
\end{equation} 
Then, we have the following key claim.\\\\
\textbf{Claim.} For $\varepsilon>0$ small enough, for each $p$ in the $\varepsilon$-approximate curve $\gamma$ of the $\varepsilon$-tube $N$ and every $t\in [\bar{t}-\tau, \bar{t}]$, the flow $\mathcal{M}$ has a strong $\varepsilon_{1}$-neck with center $p$ and radius
\begin{equation}\label{r(t1)}
r(t)=\sqrt{2(n-1)(t_{p}-t)} 
\end{equation}
at time $t$.  Here \begin{equation}
    t_{p}=\bar{t}+\frac{1}{2(n-1)}r^{2}_{p},
\end{equation} 
where $r_{p}$ is the radius of the neck with center $p\in \gamma$ at time $\bar{t}$ as above.
\begin{proof}[{Proof of claim}]
  Suppose towards contradiction, for some fixed $p$, $t_{0}\in [\bar{t}-\tau, \bar{t}]$ is the largest time such that $\mathcal{M}$ does not have a strong $\varepsilon_{1}$-neck with center $p$ and radius $r(t_{0})$ at time $t_{0}$.\\\\
  Because $\varepsilon>0$ is much smaller than $\varepsilon_{1}>0$, and $\mathcal{M}$ has a very strong $\varepsilon$-neck with center $p$ and radius $r(\bar{t})$ at time $\bar{t}$, we have $t_{0}<\bar{t}$. More precisely, let $\mathcal{T}=\mathcal{T}(2\varepsilon_{1}, \frac{1}{2}\varepsilon_{1})$ be the constant from Proposition \ref{t} (backwards stability). Because $p$ is on the central curve of the $\varepsilon$-tube $N\subset M_{\bar{t}}$, the rescaled surface $r_{p}^{-1}(M_{\bar{t}+r^2_{p}s}-p)$ is $\varepsilon$-close to the  surface
 $\sqrt{1-2(n-1)s}S^{n-1}\times \mathbb{R} $ in $B_{\varepsilon^{-1}}(0)$ for $s\in (-2\mathcal{T},0)$. Rescaling  back it, we infer that the flow 
 $\mathcal{M}$ has a strong  $\varepsilon_{1}$-neck with center $p$ and radius $r(t)$ at time $t$, provided that
 \begin{equation}\label{tubein}
     t-r^2(t)\geq \bar{t}-(2\mathcal{T}-1)r^2_{p}.
 \end{equation}
 Hence, the inequality \eqref{tubein} reverses at time $t_{0}$.  This implies that 
 \begin{equation}\label{t0}
     \frac{t_{p}-t_{0}}{t_{p}-\bar{t}}\geq
     \frac{4(n-1)(\mathcal{T}-1)+1}{2(n-1)+1}> \frac{3}{2}\mathcal{T}.
 \end{equation}
  Thus, by the intermediate value theorem, there is some $t_{1}\in (t_{0}, \bar{t})$, such that
  \begin{equation}\label{t1}
      \frac{t_{p}-t_{0}}{t_{p}-{t_{1}}}=\frac{3}{2} \mathcal{T}.
  \end{equation}
  By the definition of $t_{0}$ and since $t_{1}>t_{0}$, we know that the flow $\mathcal{M}$ has a strong $\varepsilon_{1}$-neck with center $p$ and radius $r(t_{1})$ at time $t_{1}$. Let $x_{t_{1}}$ be a point in the central sphere of this neck. We have
  \begin{equation}\label{rh}
      \left|\frac{r(x_{t_{1}})H(x_{t_{1}})}{n-1}-1\right|<\varepsilon_{1},
  \end{equation}
  and
  \begin{equation}\label{xp}
      |x_{t_{1}}-p|<(1+\varepsilon_{1})r(t_{1}).
  \end{equation}
  Also, by \eqref{r_p} and definition of $\tau$ in \eqref{tau'}, we have 
  \begin{equation}\label{rrho}
 r^2(\bar{t}-\tau)=r_{p}^2+2(n-1)\tau\leq \frac{1}{2}\rho^2.
  \end{equation}
Using this and the fact that $p$ is a neck point,  we see that 
  \begin{equation}\label{Arho}
        r(t_{1})<\frac{\sqrt{2}\rho}{2}<\frac{\delta}{2}
        \quad \text{and} \quad
       R(x_{t_{1}},t_{1})<{\rho},
  \end{equation}
provided $\varepsilon_{1}>0$ is small enough.\\\\
Because $p\in \Omega_{M}$, we have $|p-x_{i}|\geq \delta$, where $x_{i}$ denotes the location of conical singularities in Theorem \ref{decomposition}. Combining this with the inequality \eqref{xp}, we see that 
\begin{equation}
    |x_{t_{1}}-x_{i}|>\frac{\delta}{2},
\end{equation}
for all $i=1, \dots, m$, as long as $\varepsilon_{1}>0$ is small enough.
This together with \eqref{Arho} and the decomposition \eqref{MCL} in Theorem \ref{decomposition} implies that 
\begin{equation}
x_{t_{1}}\in \Omega_{M}.    
\end{equation}
Hence, by (i)(b) in Theorem \ref{decomposition}, the flow $\mathcal{M}'$ obtained from $\mathcal{M}$ by translating $(x_{t_{1}},t_{1})$ to the space-time origin and parabolically rescaling by $H(x_{t_{1}})$ is $\varepsilon$-close to a flow $\mathcal{N}$ that is either a cylindrical flow or a translating bowl.\\\\
Since $\mathcal{M}$  has a strong $\varepsilon_{1}$-neck with center $p$ and radius $r(t_{1})$ at time $t_{1}$, taking also into account \eqref{rh} and \eqref{xp}, we see that $\mathcal{N}$ has a strong $2\varepsilon_{1}$-neck with center $0$ and radius $n-1$ at time $0$. Applying Proposition \ref{t} (backwards stability) on $\mathcal{N}$, rescaling $\mathcal{N}$ by $H^{-1}(x_{t_{1}})$ and using again \eqref{rh} and \eqref{xp}, we infer that  $\mathcal{M}$ has a strong $\varepsilon_{1}$-neck with center $p$ and radius $r(t)$ at time $t$ as long as 
 \begin{equation}\label{tt1}
t\leq t_{1}-\frac{r^2(t_{1})}{2(n-1)}(\mathcal{T}-1).
 \end{equation} 
 On the other hand, using \eqref{t1} and the definition of $r(t_{1})$ in \eqref{r(t1)}, we see that
 \begin{equation}
     t_{0}-t_{1}=(1-\frac{3}{2}\mathcal{T})(t_{p}-t_{1})\leq (1-\mathcal{T})(t_{p}-t_{1})=-\frac{r^2(t_{1})}{2(n-1)}(\mathcal{T}-1),
 \end{equation} so $t_{0}$ satisfies \eqref{tt1}. 
 Hence, for $\varepsilon>0$ small enough, $\mathcal{M}$ has a strong $\varepsilon_{1}$-neck at $p$ with radius $r(t_{0})$ at time $t_{0}$. This contradiction completes the proof of the claim.
  \end{proof}
\noindent We continue proving Theorem \ref{main}. By the claim, we can apply Proposition \ref{st} (small axis tilt)  and obtain that for every $p\in \gamma$, there exists a fixed $O_{p}\in SO(n+1)$, such that for all $t\in [\bar{t}-\tau, \bar{t}]$ we have that $M_{t}$ is $\varepsilon_{0}$-close to the cylinder 
\begin{equation}
Z_{p}=p+O_{p}\left(\sqrt{2(n-1)(t_{p}-t)}S^{n-1}\times \mathbb{R}\right) 
\end{equation}
in $B_{\varepsilon_{0}^{-1}\sqrt{2(n-1)(t_{p}-t)}}(p)$.    
\\\\
Furthermore, as long as $|p_{1}-p_{2}|<(4\varepsilon_{0})^{-1}\sqrt{\tau}$, the associated cylinders $Z_{p_{1}}$ and $Z_{p_{2}}$ will align up to an $\varepsilon_{0}$-error rotation. Then, $\|O_{p_{1}}-O_{p_{2}}\|=O(\varepsilon_{0})$ and $|t_{p_{1}}-t_{p_{2}}|=O({\varepsilon_{0}})$. This implies that the  intrinsic distance is controlled by extrinsic distance, namely,
\begin{equation}
d_{\gamma}(p_{1}, p_{2})\leq (1+O(\varepsilon_{0}))|p_{1}-p_{2}|.   
\end{equation}
for any two points $p_{1}, p_{2} \in \gamma$ with $|p_{1}-p_{2}|<(4\varepsilon_{0})^{-1}\sqrt{\tau}$.\\\\
Noticing that $\tau$ only depends on $\rho=\rho(\varepsilon)$, $\delta=\delta(\varepsilon)$, and $\Omega$ is covered by controlled number of balls of radius $(4\varepsilon_{0})^{-1}\sqrt{\tau}$, we conclude that $L(M_{\bar{t}})$ is uniformly bounded in $(T-\delta^2, T)$. This completes the proof.
\end{proof}
\begin{proof}[Proof of Theorem \ref{cb}]
 We want to prove that there is a constant $C<\infty$, such that for $t\in [0, T)$, we have
\begin{equation}
 \int_{M_{t}}|A|d\mu_{t} <C.
\end{equation}
Using the decomposition \eqref{MCL}  in  Theorem \ref{decomposition}, we only need to verify the curvature bound in $\Omega_{L}$, $\Omega_{M}$ and $\Omega_{C}$.  Since the flow has bounded curvature in $\Omega_{L}$, the curvature estimation holds in $\Omega_{L}$. Hence, we only need to show curvature estimation in $\Omega_{M}$ and $\Omega_{C}$. \\\\
For the flow restricted in the mean convex part $\Omega_{M}$, by the description of $\Omega_{M}$ in Theorem \ref{decomposition}, the flow in $\Omega_{M}$ can be decomposed into the union of controlled number of $\varepsilon$-tubes possibly with caps as ends or identified ends. Hence, we only need to estimate curvature bound on each $\varepsilon$-tube with their possibly cap ends. \\\\
Notice that the curvature of the flow in $\Omega_{M}$ is bounded below by some number $\bar{H}=c\rho^{-1}>0$, where $\rho$ is from Theorem \ref{decomposition} and $c>0$. This implies that these caps only contribute $C(n,\varepsilon)\rho$ amount to the curvature integral. On the other hand, for each $\varepsilon$-tube $N\subset M_{t}$, by Vitali's covering lemma, we can write $N$ as the union of a maximal collection of $\varepsilon$-necks $N_{j}$ with center $p_{j}$ and radius $r_{j}$ at time $t$, such that the collection of balls $\{B_{\frac{r_{j}}{5}}\}$ are disjoint. Hence, by Theorem \ref{main}, we have 
\begin{equation}
    \int_{N} |A|^{n-1} d\mu_{t}\leq \sum_{j} \int_{N_{j}} |A|^{n-1} d\mu_{t} \leq C\sum_{j}r_{j}\leq 5C \sup_{{t\in [0, T)}} \text{diam}(M_{t},d_{t})<\infty,
\end{equation}
where $C=C(n, \varepsilon)<\infty$.\\\\
For the flow restricted in the conical part $\Omega_{C}$, we estimate the curvature integral in every conical neighborhood $B_{\delta}(x_{i})\times (T-\delta^2, T)$, where $i\in\{1, \dots, m\}$. Because $M_{t}\cap B_{\delta}(x_{i})$ can be written as $C^{l+1}$ graph over   $\left(\sqrt{T-t}\Sigma_{i}+x_{i}\right)\cap B_{\delta}(x_{i})$ with $C^{l+1}$ norm less than $\varepsilon$, where $\Sigma_{i}$ is an asymptotically conical shrinker, we can find a constant $C=C(n, \varepsilon, \delta)<\infty$, such that
\begin{equation}
\int_{M_{t}\cap B_{\delta}(x_{i})}|A|^{n-1}d\mu_{t}\leq C(n,\varepsilon)\int_{\left(\sqrt{T-t}\Sigma_{i}+x_{i}\right)\cap B_{\delta}(x_{i})} |A|^{n-1} d\mu^{i}_{t}<C.
\end{equation}
This completes the proof of Theorem \ref{cb}.
\end{proof}

\bigskip

\bibliographystyle{alpha}
\bibliography{references}
\bigskip
\textit{Department Of Mathematics, University of Toronto, Toronto, ON, M5S 2E4, Canada\\}
\textit{E-mail Address: wenkui.du@mail.utoronto.ca}

\end{document}